\input xy
\xyoption{all}

\documentclass{tac}
\usepackage{amsmath,latexsym,amssymb,mathrsfs,lipsum}
\usepackage{hyperref}
\usepackage{color}

\usepackage{textcomp}

\newtheorem{theorem}{Theorem}
\newtheorem{lemma}{Lemma}

\newtheorem{remark}{Remark}

\newtheorem{proposition}{Proposition}

\newtheorem{corollary}{Corollary}

\newcommand{\Eq}{\ensuremath{\mathsf{Eq}(\mathbb C)}}

\newcommand{\Par}{\ensuremath{\mathsf{ParOrd}(\mathbb C)}}

\newcommand{\C}{\mathbb C}
\newcommand{\N}{\mathcal N}

\def\theenumi{\alph{enumi}}
\def\Alex{\mathsf{Alex}}
\def\KoAlex{\mathsf{T_0Alex}}
\def\PartAlex{\mathsf{PartAlex}}
\begin{document}

\title{A new Galois structure in the category of internal preorders}
\author{Alberto Facchini, Carmelo Finocchiaro and Marino Gran}
\keywords{Internal preorders, partial orders, Galois theory, monotone-light factorization system, Alexandroff-discrete spaces. \\ \noindent 2020 Mathematics Subject Classification: 18E50, 18A32, 18B35,  18E40, 06A15 }

\thanks{The first author was partially 
	supported by Ministero dell'Istruzione, dell'Universit\`a e della 
	Ricerca (Progetto di ricerca di rilevante interesse nazionale 
	``Categories, Algebras: Ring-Theoretical and Homological Approaches 
	(CARTHA)'') and Dipartimento di Matematica ``Tullio Levi-Civita'' of 
	Universit\`a di Padova (Research program DOR1828909 ``Anelli e categorie 
	di moduli''). The second author was partially supported by GNSAGA of Istituto Nazionale di Alta Matematica, by Dipartimento di Matematica ``Tullio Levi-Civita'' of Universit\`a di Padova (Research program DOR1828909 ``Anelli e categorie di moduli'') and by Dipartimento di Matematica e Informatica of Universit\`a di Catania (Research program ``Propriet\`a algebriche locali
	e globali di anelli associati a curve e ipersuperfici" PTR 2016-18). The first and the second authors were partially supported by the research program ``Reducing complexity in algebra, logic, combinatorics - REDCOM'' ``Ricerca Scientifica di Eccellenza 2018'' Fondazione Cariverona. The third author was supported by a ``Visiting scientist scholarship - year 2019'' by Universit\`a di Padova, and by a collaboration project \emph{Fonds d'Appui \`a l'Internationalisation} of the Universit\'e catholique de Louvain.}  

\copyrightyear{2019}\maketitle

\begin{abstract} Let $\mathsf{PreOrd}(\mathbb C)$ be the category of internal preorders in an exact category $\mathbb C$. We show that the pair $(\mathsf{Eq}(\mathbb C), \mathsf{ParOrd}(\mathbb C))$ is a pretorsion theory in $\mathsf{PreOrd}(\mathbb C)$, where $\mathsf{Eq}(\mathbb C)$ and $\mathsf{ParOrd}(\mathbb C)$) are the full subcategories of internal equivalence relations and of internal partial orders in $\mathbb C$, respectively. We observe that $\mathsf{ParOrd}(\mathbb C)$ is a reflective subcategory of $\mathsf{PreOrd}(\mathbb C)$ such that each component of the unit of the adjunction is a pullback-stable regular epimorphism. The reflector $F:\mathsf{PreOrd}(\mathbb C)\to \mathsf{ParOrd}(\mathbb C)$ turns out to have stable units in the sense of Cassidy, H\'ebert and Kelly, thus inducing an admissible categorical Galois structure. In particular, when $\mathbb C$ is the category $\mathsf{Set}$ of sets, we show that this reflection induces a monotone-light factorization system (in the sense of Carboni, Janelidze, Kelly and Par\'e) in $\mathsf{PreOrd}(\mathsf{Set})$. A topological interpretation of our results in the category of Alexandroff-discrete spaces is also given, via the well-known isomorphism between this latter category and $\mathsf{PreOrd}(\mathsf{Set})$.
\end{abstract}

\section*{Introduction}
The category $\Par$ of internal partial orders in an exact \cite{Barr} category $\C$ is reflective in the category ${\mathsf{PreOrd}(\mathbb C) \, }$ of internal preorders in $\C$:
\begin{equation}\label{adjunction1}
\xymatrix@=30pt{
{\mathsf{PreOrd}(\mathbb C) \, } \ar@<1ex>[r]_-{^{\perp}}^-{F} & {\, \Par.\, }
\ar@<1ex>[l]^U  }
 \end{equation}
Here $U \colon \Par \rightarrow \mathsf{PreOrd}(\mathbb C)$ is the forgetful functor, and $F \colon \mathsf{PreOrd}(\mathbb C) \rightarrow \Par$ is the reflector, which has a simple description (given in Section $1$), and is such that each component of the unit of the adjunction \eqref{adjunction1} is a pullback-stable regular epimorphism in $\mathsf{PreOrd}(\mathbb C)$. 

We first show that the subcategory $\Par$ can be seen as a torsion-free subcategory for a non-abelian torsion theory in the category $\mathsf{PreOrd}(\mathbb C)$. 
In order to define a notion of torsion theory in a category which does not have a zero object, such as $\mathsf{PreOrd}(\mathbb C)$, the notion of short exact sequence has to be made more flexible, as we now recall.
 
 There is a natural notion of short $\N$-exact sequence, which is defined relatively to an \emph{ideal $\N$ of morphisms}  in the sense of Ehresmann \cite{Ehr64} (see also \cite{Mantovani, GJ-2019}, and the references therein).
 
By an ideal $\N$ of morphisms in a category $\mathbb A$ is meant a distinguished class of morphisms in $\mathbb A$ with the following property: for any composable pair of morphisms 
\begin{equation}\label{composable}
\xymatrix{ A \ar[r]^f & B \ar[r]^g &C }\end{equation}
in $\mathbb A$, if either $g$ or $f$ belongs to $\N$, then the composite $gf$ belongs to $\N$. 
When $\mathbb A$ is a category and $\mathcal Z$ a (non-empty, full and replete) subcategory of $\mathbb A$, we write $\N(A, B)$ for the set of morphisms (in $\mathbb A$) from $A$ to $B$  which factor through an object of $\mathcal Z$. This class of morphisms obviously forms an ideal $\N$ in $\mathbb A$. The objects of $\mathcal Z$, which fully determine $\N$, are called the \emph{trivial objects} of $\mathbb A$.
 
 Given a morphism $f\colon A\to B$ in $\mathbb A$, one says that $\kappa \colon K \to A$ in $\mathbb A $ is an \emph{$\N$-kernel} of $f$ (or a $\mathcal Z$\emph{-prekernel} of $f$, following \cite{FF, pretorsion}), if $f\kappa \in \mathcal N$ and, moreover,
whenever $\lambda \colon X\to A$ is a morphism in $\mathbb A$ with $f\lambda \in \N$, there is a unique morphism $\lambda'\colon X \to K$ in $\mathbb A$ such that $\lambda=\kappa \lambda'$. 

Of course, if $\mathcal Z = \{ 0\}$, so that $\mathcal Z$ is reduced to the zero-object $0$ of $\mathbb A$ (when this latter exists), the notion of $\N$-kernel gives back the usual notion of kernel. The notion of $\N$-cokernel is defined dually, and we then get the needed general notion of short $\N$-exact sequence (also called $\mathcal Z$-preexact sequence in \cite{FF, pretorsion}).
Indeed, a pair of composable arrows as in diagram \eqref{composable} is a \emph{short $\N$-exact sequence} if $f$ is the $\N$-kernel of $g$, and $g$ is the $\N$-cokernel of $f$.

A \emph{pretorsion theory} $(\mathcal T, \mathcal F)$ in $\mathbb A$ is a pair of (full replete) subcategories of $\mathbb A$ with the following properties:
if $\N$ is the ideal of morphisms associated with the full replete subcategory $\mathcal Z : = \mathcal T \wedge \mathcal F$ of \emph{trivial objects}, then
\begin{enumerate}
\item any morphism $T \rightarrow F$ from an object $T \in \mathcal T$ to an object $F \in \mathcal F$ belongs to $\N$, i.e. it factors through an element in $\mathcal Z$;
\item for any object $A \in \mathbb A$ there is a short $\N$-exact sequence $$\xymatrix{ T(A)  \ar[r]^k & A \ar[r]^p &F(A) }$$
with $T(A) \in \mathcal T$ and $F(A) \in \mathcal F$.
\end{enumerate}
This general notion of pretorsion theory \cite{FF, pretorsion} includes the classical one in the pointed case (\cite{Dickson, BG}), and is in particular suitable for our study of the reflection \eqref{adjunction1} above. Indeed, in that example the pretorsion theory is given by the pair of (full replete) subcategories $(\Eq,\Par)$ in the category $\mathsf{PreOrd}(\mathbb C)$, where $\Eq$ is the category of internal equivalence relations in $\C$ (see Lemma \ref{N-torsion}). The ideal $\N$ of morphisms consists then in the morphisms in $\mathsf{PreOrd}(\mathbb C)$ that factor through the objects in $\mathcal Z = \Eq \wedge \Par = \mathsf{Dis}(\mathbb C)$, where $\mathsf{Dis}(\mathbb C)$ is the (full replete) subcategory of $\mathsf{PreOrd}(\mathbb C)$ whose objects are the \emph{discrete equivalence relations} in $\C$.
The functor $F \colon \mathsf{PreOrd}(\mathbb C) \rightarrow \Par$ in \eqref{adjunction1} is the reflector to the torsion-free subcategory $\Par$ of this pretorsion theory in $\mathsf{PreOrd}(\mathbb C)$.
In Theorem \ref{semi-localization} this reflector is shown to have \emph{stable units} (in the sense of \cite{CHK}), 
and the corresponding factorization system $(\mathcal E, \mathcal M)$, where $\mathcal E$ is the class of morphisms inverted by $F$ and $\mathcal M$ the class of \emph{trivial coverings} for the induced Galois structure (in the sense of \cite{Janelidze-Galois}) is described in Proposition \ref{factor-E}. Note that any full reflective subcategory with a reflector having stable units always determines an admissible Galois structure (with respect to the classes of \emph{all} morphisms in both the categories, see \cite{CJKP} for more details).
The class ${\mathcal M}^*$ of \emph{coverings} for this Galois theory is then described in the special case where ${\mathbb C}$ is the exact category $\mathsf{Set}$ of sets (Proposition \ref{character.covering}). The class ${\mathcal M}^*$ is also part of a \emph{monotone-light factorization system}  $({\mathcal E}', {\mathcal M}^*)$ (in the sense of \cite{CJKP}), where ${\mathcal E}'$ is the subclass of $\mathcal E$ whose elements are the morphisms in $\mathcal E$ which are pullback stable (Proposition \ref{monotone-light}).

 It is well known that the category $\mathsf{PreOrd(Set)}$ of internal preorders in the category of sets is isomorphic to the category of \emph{Alexandroff-discrete} topological spaces. Thanks to this topological interpretation the monotone-light factorization system for $\mathsf{PreOrd(Set)}$ admits an alternative description, given in Proposition \ref{monotone-light-top}. In particular, the \emph{coverings} induced by the Galois structure mentioned above are precisely the continuous maps that have $T_0$ fibres.
   The fact that the reflector $F:\mathsf{PreOrd}(\mathsf{Set})\to \mathsf{ParOrd}(\mathsf{Set})$ has stable units was already mentioned in Xarez's thesis \cite{Xarez-Thesis}, where some monotone-light factorization systems in the category $\mathsf{Cat}$ of (small) categories were also investigated. \\
 
{ \bf Acknowledgement.} The authors would like to thank to the anonymous referee for several useful suggestions on the first version of the paper.
 
\section{Internal preorders in an exact category.}
In this section $\mathbb C$ will denote an exact category \cite{Barr}. This means that:
\begin{itemize}
\item$\C$ is a finitely complete category;
\item any morphism $f \colon A \rightarrow B$ in $\C$ has a (unique) factorization $f = i p$
$$\xymatrix{A \ar[rr]^{f} \ar@{>>}[dr]_p  & &  B \\
&{I\,\,  } \ar@{>->}[ur]_i & 
}$$
where $p$ is a regular epimorphism and $i$ is a monomorphism (the subobject $i : I \rightarrow B$ will be called the \emph{regular image} of $f$);
\item regular epimorphisms are pullback-stable: in any pullback
$$
\xymatrix{E {\times}_B A \ar@{>>}[r]^-{p_2}
 \ar[d]_{p_1} & A \ar[d]^f \\
 E \ar@{>>}[r]_p & B
}
$$
the morphism $p_2$ is a regular epimorphism whenever $p$ is so;
\item any (internal) equivalence relation $\langle r_1,r_2 \rangle \colon {R} \rightarrow {X \times X}$ on an object $X$ in $\C$ is \emph{effective}, i.e., $R$ is the kernel pair  
$$
\xymatrix{R \ar[r]^-{r_2} \ar[d]_{r_1}& X \ar[d]^f \\
X \ar[r]_f & Y.
}
$$
of some arrow $f \colon X \rightarrow Y$ in $\C$. The kernel pair of $f$ will be denoted by $Eq(f)$.
\end{itemize}
Any variety of universal algebras is an exact category, in particular the varieties of (pointed) sets, semigroups, monoids, commutative monoids, groups, abelian groups, Lie algebras, (commutative) rings, etc. Any elementary topos is an exact category, as is also any abelian category. The categories of $\mathsf{C}^*$-algebras and the category of compact Hausdorff spaces are also exact. It is well known that the category $\mathsf{Top}$ of topological spaces is not exact, since regular epimorphisms are not pullback-stable.
 
Given two relations $\langle r_1,r_2 \rangle \colon {R} \rightarrow {X \times Y}$ and $\langle s_1,s_2 \rangle  \colon {S} \rightarrow {Y \times Z}$ in an exact category, their relational composite ${S\circ R} \rightarrow {X \times Z}$ can be defined as follows: take the pullback 
$$\xymatrix{R \times_Y S \ar[r]^-{p_2} \ar[d]_{p_1} & S \ar[d]^{s_1}\\
	R \ar[r]_{r_2} & Y
}
$$
of $r_2$ and $s_1$, and then the regular image $S \circ R$ of the arrow $\langle r_1p_1,s_2 p_2\rangle$:
$$
\xymatrix{R\times_Y S \ar[rr]^-{\langle r_1p_1,s_2 p_2\rangle} \ar@{>>}[dr] & & X\times Z. \\
& {S\circ R \quad} \ar@{>->}[ur]}
$$
In the category $\mathsf{Set}$ of sets one clearly has $$S \circ R = \{ (x,z) \in X \times Z \, \mid \, \exists y \in Y, (x,y)\in R, (y,z) \in S \},$$i.e. the usual composite of the relations $S$ and $R$. It is well known that the composition of relations is associative when $\mathbb C$ is a regular category, thus in particular when $\mathbb C$ is an exact category.

Consider then the category $\mathsf{PreOrd}(\mathbb C)$ of (internal) preorders in $\C$. An object $(A, \rho)$ in $\mathsf{PreOrd}(\mathbb C)$ is simply a relation $\langle r_1,r_2 \rangle \colon \rho \rightarrow A \times A$ on $A$ that is \emph{reflexive}, i.e. it contains the ``discrete relation'' $\langle 1_A,1_A \rangle \colon A \rightarrow A \times A$ on $A$ - the ``equality relation'', denoted by $\Delta_A$ -
and $\rho$ is \emph{transitive}: there is a
morphism $\tau \colon \rho \times_A \rho \rightarrow \rho$ such that $r_1 \tau =  r_1 p_1$ and $r_2 \tau =  r_2 p_2$, where $\rho \times_A \rho$ is the ``object part'' of the pullback
$$
\xymatrix{ \rho \times_A \rho \ar[r]^-{p_2} \ar[d]_-{p_1}& \rho \ar[d]^{r_1} \\
\rho \ar[r]_{r_2} & A.
}
$$ 
 
A morphism $(A, \rho) \rightarrow  (B, \sigma)$ in the category $\mathsf{PreOrd}(\mathbb C)$ is a pair of morphisms $(f,\overline{f})$ in $\C$ making the following diagram commute:
$$\xymatrix{  \rho \ar@<.5ex>[d]^{r_2} \ar@<-.5ex>[d]_{r_1}  \ar[r]^{\overline{f}}  & \sigma \ar@<.5ex>[d]^{s_2} \ar@<-.5ex>[d]_{s_1}  \\
 A \ar[r]_{f} & {B.} &
}
$$
By this we mean that $f r_1= s_1 \overline{f} $ and $f r_2= s_2 \overline{f}$ (this convention will occur often in the sequel).
The category $\Eq$ is the full subcategory of $\mathsf{PreOrd}(\mathbb C)$ whose objects are (internal) \emph{equivalence relations} in $\C$, that is, those preorders $(A, \rho)$ in $\C$ that are also \emph{symmetric}: this means that there is a morphism $\gamma \colon \rho \rightarrow \rho$ with $r_1 \gamma = r_2$ and $r_2 \gamma = r_1$. Equivalently, one can ask that $\rho^o = \rho$, where $\rho^o = \langle r_2, r_1 \rangle \colon \rho \rightarrow A \times A$ is the opposite relation of $\rho = \langle r_1, r_2 \rangle \colon \rho \rightarrow A \times A$.
$\Par$ will denote the category of (internal) \emph{partial orders} in $\C$, where a preorder $(A, \rho)$ is a partial order when, moreover, the relation $\rho$ is also \emph{antisymmetric}. This property can be expressed by the fact that the subobject $\sim_{\rho}  = \rho \wedge \rho^o$ obtained as the diagonal of the pullback
\begin{equation}\label{sim}
\xymatrix{ \sim_{\rho} \ar[r]^-{} \ar[d]_-{} & \rho^o \ar[d]^{\langle r_2, r_1 \rangle} \\
\rho \ar[r]_-{\langle r_1, r_2 \rangle }  & A \times A,}
\end{equation}
 is the discrete relation $(A, \Delta_A)$ on $A$. 
\begin{remark}\label{equivalence-relation}
\emph{ Given \emph{any} preorder $(A, \rho)$, the relation $\sim_{\rho}$ defined by the pullback \eqref{sim} is an equivalence relation: indeed, $\sim_{\rho}$ is obviously reflexive and transitive, and  it is
symmetric since
$$\sim_{\rho}^o  = (\rho \wedge \rho^o)^o=  \rho^o \wedge \rho = {\sim}_{\rho}.$$ }
\end{remark}
\begin{remark}\label{equivalence}
\emph{It is clear that a preorder $(A, \rho)$ belongs to both $\Eq$ and $\Par$ if and only if the order relation $\rho$ is the discrete relation $ \Delta_A$ on $A$, since in that case $\Delta_A = {\sim}_{\rho} = \rho \wedge \rho^o = \rho$. If we write $\mathsf{Disc}(\C)$ for the full subcategory of $\mathsf{PreOrd}(\mathbb C)$ whose objects are the preorders $(A, \Delta_A)$ equipped with the discrete relation, then  $$\Eq \wedge \Par = \mathsf{Dis}(\C).$$ }
\end{remark}
In this section $\N$ will always denote the ideal of morphisms in $\mathsf{PreOrd}(\mathbb C)$ that factor through its full subcategory $\mathsf{Dis}(\C)$ of discrete relations. 
\begin{lemma}\label{N-torsion}
The pair of full subcategories $(\Eq, \Par)$ is a pretorsion theory in $\mathsf{PreOrd}(\mathbb C)$.
\end{lemma}
\begin{proof}
Let us first prove that any $(f, \overline{f}) \colon (A, R) \rightarrow (B, \sigma)$ in $\mathsf{PreOrd}(\mathbb C)$
$$\xymatrix{  & R \ar@{.>}[dl] \ar@<.5ex>[d]^{r_2} \ar@<-.5ex>[d]_{r_1}  \ar[r]^{\overline{f}}  & \sigma \ar@<.5ex>[d]^{s_2} \ar@<-.5ex>[d]_{s_1}  \\
 Eq(f)  \ar@<.5ex>[r]^{} \ar@<-.5ex>[r]_{} &  A \ar[r]_{f} & {B} &
}
$$
where $(A, R) \in \Eq$ and $(B, \sigma) \in \Par$, is a morphism in $\N$.
For this one observes that $$f(R) = f(R \wedge R^o) \le f(R) \wedge f(R)^o \le \sigma \wedge (\sigma)^o = {\sim}_{\sigma} =\Delta_B,$$
where $R= R \wedge R^o$ since $R$ is symmetric, and the last equality follows from the fact that $(B, \sigma)$ is a partial order. 

Next, given any preorder $(A, \rho)$ on $A$, build the canonical quotient $\pi \colon A \rightarrow \frac{A}{\sim_\rho}$ (the coequalizer of $p_1$ and $p_2$), that exists since the category $\C$ is exact. In the diagram 
 $$
 \xymatrix{    &\rho  \ar@{->>}[r]^-{\overline{\pi}}  \ar@<.5ex>[d]^{r_2} \ar@<-.5ex>[d]_{r_1}  & \pi(\rho) \ar@<.5ex>[d]^{} \ar@<-.5ex>[d]_{} & \\
{\sim_{\rho}}  \ar@<.5ex>[r]^{p_1} \ar@{.>}[ur]^{i} \ar@<-.5ex>[r]_{p_2}  & A \ar@{->>}[r]_{\pi} & {\frac{A}{\sim_\rho}} &
}$$
the dotted arrow $i$ exists since $\sim_{\rho}\le \rho$ by definition of $\sim_{\rho}$, 
whereas $\pi(\rho)$ is the regular image of $\rho$ along $\pi$.
The induced relation $\pi (\rho)$ on ${\frac{A}{\sim_\rho}} $ is reflexive and transitive. Indeed, it is clear that $\pi (\rho)$ is reflexive and, by using the formula $\pi (\rho)= \pi \circ \rho \circ \pi^o$ giving the regular image of the relation $\rho$ along the regular epimorphism $\pi$, we see that $\pi(\rho)$ is transitive: 
$$(\pi \circ \rho \circ \pi^o) \circ (\pi \circ \rho \circ \pi^o) = \pi \circ \rho \circ (\pi^o \circ \pi) \circ \rho \circ \pi^o = 
\pi \circ \rho \, \circ \sim_{\rho} \circ \, \rho \circ \pi^o \le \pi \circ \rho \circ \pi^o.$$
 The formula for the inverse image along $\pi$ gives the inequality $$\pi^{-1} (\pi( \rho) )= \pi^o \circ \pi \circ \rho \circ \pi^o \circ \pi = {\sim}_{\rho} \circ \rho \, \circ  \sim_{\rho} \le \rho,$$
showing that $\rho = \pi^{-1} (\pi( \rho))$.
Diagrammatically, this means that the following square, yielding the regular image factorization, is a pullback:
\begin{equation}\label{direct-inverse}
\xymatrix{ {\rho} \ar@{>>}[r]^-{\overline{\pi}} \ar[d]_-{\langle r_1, r_2 \rangle} & \pi(\rho) \ar[d]^{} \\
A \times A  \ar@{>>}[r]_-{\pi \times \pi}  & \frac{A}{\sim_{\rho}} \times \frac{A}{\sim_{\rho}}.}
\end{equation}
We can now show that $\pi(\rho)$ is antisymmetric: indeed, in the cube 
$$
  \vcenter{\xymatrix@1@R=15pt@C=15pt{
    \sim_{\rho}   \ar@{..>}[rr]^-{\tilde{\pi}}\ar@<-2pt>[dd] \ar[dr] & &  \sim_{\pi(\rho)} \ar@<-2pt>[dd]\ar[dr]\\
    & \rho^o \ar@<-2pt>[dd] \ar@{>>}[rr]^(.5){\overline{\pi}} & & \pi(\rho^o) = \pi(\rho)^o \ar@<-2pt>[dd] \\
    \rho  \ar@{>>}[rr]_(.7){\overline \pi} \ar[dr] & & \pi(\rho)   \ar[dr]  \\
    & A \times A  \ar@{>>}[rr]_{\pi \times \pi } & &  \frac{A}{\sim_{\rho}} \times \frac{A}{\sim_{\rho}} }}
$$
the left-hand square and the bottom square are both pullbacks, and the commutativity of the diagram implies that the external rectangle here below is a pullback:
$$\xymatrix{
 \sim_{\rho}  \ar@{..>}[r]^-{\tilde{\pi} }  \ar[d] & \sim_{\pi(\rho)} \ar[r] \ar[d] & \pi (\rho) \ar[d] \\
 \rho^o \ar@{->>}[r]_-{\overline{\pi}} & \pi(\rho)^o \ar[r] & \frac{A}{\sim_{\rho}} \times \frac{A}{\sim_{\rho}}
}
$$
Since the right-hand square is a pullback, we deduce that the left-hand square is a pullback as well, and the induced arrow $\tilde{\pi}$ is then a regular epimorphism, since so is $\overline{\pi}$. It follows that $\sim_{\pi(\rho)} = {\pi} ( \sim_{\rho} ) = \Delta_{\frac{A}{\sim_{\rho}}} $, and the relation $\sim_{\pi(\rho)}$ is antisymmetric. 

For any preorder $(A, \rho)$, the canonical short $\N$-exact sequence is then given by the diagram 
\begin{equation}\label{Nsee}
 \xymatrix{ {\sim_{\rho}}  \ar[r]^i  \ar@<.5ex>[d]^{p_2} \ar@<-.5ex>[d]_{p_1}  &\rho  \ar@{->>}[r]^-{\overline{\pi}}  \ar@<.5ex>[d]^{r_2} \ar@<-.5ex>[d]_{r_1}  & \pi(\rho) \ar@<.5ex>[d]^{} \ar@<-.5ex>[d]_{} & \\
 A  \ar@{=}[r] & A \ar@{->>}[r]_{\pi} & {\frac{A}{\sim_\rho}} &
}\end{equation}
where $(A, {\sim_{\rho}})$ is an equivalence relation by Remark \ref{equivalence-relation}, and $({\frac{A}{\sim_\rho}}, \pi(\rho))$ is a partial order. To see that $(\pi, \overline{\pi})$ is the $\N$-cokernel of $(1_A, i)$, observe that by definition of $\pi$ the composite $(\pi, \overline{\pi})(1_A, i) $ belongs to $\N$. One then considers a morphism $(g, \overline{g}) \colon (A, \rho) \rightarrow (B, \sigma)$ in $\N$, so that 
$g p_1 = g p_2$. The universal property of the quotient $\pi$ yields a unique $\alpha$ such that $\alpha \pi = g$ and the result then follows (one checks that there is a unique $\overline{\alpha} \colon \pi(\rho) \rightarrow \sigma$ such that $ (\alpha, \overline{\alpha}) ( \pi, \overline{\pi}) = (g, \overline{g})$). 

Dually, we now show that $(1_A, i)$ is the $\N$-kernel of $(\pi, \overline{\pi})$.
To show that $(1_A, i)$ is the $\N$-kernel of $(\pi, \overline{\pi})$,
 consider then a morphism $(f, \overline{f}) \colon (B, \tau) \rightarrow (A, \rho)$ in $\mathsf{PreOrd}(\mathbb C)$ such that $(\pi, \overline{\pi}) (f, \overline{f}) \in \N$.
There is clearly only one possible arrow $f \colon B \rightarrow A$ such that $1_A f = f $, and it remains to show that there is a (unique) $\hat{f} \colon \tau \to {}\sim_{\rho}$ making the square 
$$\xymatrix{ \tau  \ar@{.>}[r]^{\hat{f}}  \ar@<.5ex>[d]^{t_2} \ar@<-.5ex>[d]_{t_1} &\sim_{\rho}  \ar@<.5ex>[d]^{p_2} \ar@<-.5ex>[d]_{p_1} \\
B \ar[r]_{f} & A 
} $$
commute.
The fact that $(\pi, \overline{\pi}) (f, \overline{f})$ factors through a discrete equivalence relation implies that $\pi f t_1 = \pi f t_2$. The equivalence relation $\sim_{\rho}$ is the kernel pair of its coequalizer ${\pi}$ (since $\C$ is an exact category), and its universal property gives the needed morphism $\hat{f}$.
\end{proof}

\begin{corollary}\label{adj-pre-par}
There is an adjunction
\begin{equation}\label{adjunction}
\xymatrix@=30pt{
{\mathsf{PreOrd}(\mathbb C) \, } \ar@<1ex>[r]_-{^{\perp}}^-{F} & {\, \Par ,\, }
\ar@<1ex>[l]^U  }
 \end{equation}
 where $U$ is the forgetful functor and $F$ is its left adjoint. Moreover, $\Par$ is a (regular epi)-reflective subcategory of ${\mathsf{PreOrd}(\mathbb C) \, }$.  \end{corollary} 
 \begin{proof}
 We already know that $\Par $ is a reflective subcategory of ${\mathsf{PreOrd}(\mathbb C) \, }$, because it is the torsion-free subcategory in the pretorsion theory $(\Eq, \Par)$ in $\mathsf{PreOrd}(\mathbb C)$ (see Corollary $3.4$ in \cite{pretorsion}, although in this case it can be easily deduced from the arguments in Lemma \ref{N-torsion}). The morphism $(\pi, \overline{\pi})$ in the $\mathcal N$-exact sequence \eqref{Nsee} is the $(A,\rho)$-component of the unit of the adjunction \eqref{adjunction}, and this is a pullback-stable regular epimorphism in ${\mathsf{PreOrd}(\mathbb C)}$ (since the square \eqref{direct-inverse} is a pullback).
 \end{proof}

\begin{corollary}
Let $\C$ be an exact category, $A$ an object in $\C$. Then there is a bijection between:
\begin{enumerate}
\item preorder structures $(A, \rho)$ on $A$;
\item pairs $(R, \le)$ where $R$ is an equivalence relation on $A$, and $\le$ is a partial order on the quotient $\frac{A}{R}$:
\begin{equation}\label{bijection}
\xymatrix{    &  & \le \ar@<.5ex>[d]^{} \ar@<-.5ex>[d]_{} & \\
{R}  \ar@<.5ex>[r]^{} \ar@<-.5ex>[r]_{}  & A \ar[r]_{q} & {\frac{A}{R}.} &
}\end{equation}
\end{enumerate}
\end{corollary}
\begin{proof}
We proved in Lemma \ref{N-torsion} that, given a preorder $(A, \rho)$,  there is an equivalence relation $\sim_{\rho}$ on $A$ with the property that 
 the quotient $(\frac{A}{\sim_{\rho}}, \pi(\rho))$ is a partial order. The required pair $(R, \le)$ in $(b)$ is then the pair $(\sim_{\rho}, \pi(\rho))$.
 
Conversely, given a pair $(R, \le)$ as in $(b)$ and the corresponding diagram \eqref{bijection}, one can complete it by taking the inverse image $q^{-1} (\le)$ of $\le$ along $q$:
\begin{equation}\label{bijection2}
\xymatrix{    & q^{-1}(\le) \ar@{.>>}[r]^-{\overline{q}}  \ar@{.>}@<.5ex>[d]^{} \ar@{.>}@<-.5ex>[d]_{}& \le \ar@<.5ex>[d]^{} \ar@<-.5ex>[d]_{} & \\
{R}  \ar@<.5ex>[r]^{} \ar@<-.5ex>[r]_{}  & A \ar@{->>}[r]_{q} & {\frac{A}{R}.} &
}
\end{equation}
The relation $(A, q^{-1}(\le))$ is reflexive and transitive, because so is $(\frac{A}{R}, \le)$. 

On the one hand, if we apply the reflector $F$ of \eqref{adjunction} to $(A, q^{-1}(\le))$, we get the pair $(R, \le)$ again (up to isomorphism): indeed, the assumption that $(\frac{A}{R}, \le)$ is a partial order implies that
$$\sim_{q^{-1}(\le)}{} = q^{-1}(\le) \wedge q^{-1}(\le)^o = q^{-1}(\le) \wedge q^{-1}(\le^o) = q^{-1} (\le \wedge \le^o) = q^{-1} (\Delta_{\frac{A}{R}})= R,$$
and one clearly has that $q(q^{-1}(\le) ) = \le$ (see diagram \ref{bijection2}).

On the other hand, if we begin with a preorder $(A, \rho)$, the fact that the diagram \eqref{direct-inverse} is a pullback implies that $\rho = \pi^{-1} (\pi (\rho))$.
This shows that the correspondence is a bijection. 
\end{proof}
Let $\mathbb A$ be a reflective subcategory of a category $\mathbb B$ 
\begin{equation}\label{semi-loc}
\xymatrix@=30pt{
{\mathbb B \, } \ar@<1ex>[r]_-{^{\perp}}^-{G} & {\, {\mathbb A}\, }
\ar@<1ex>[l]^V  }
 \end{equation}
 with $V$ the inclusion functor and $G$ its left adjoint.
 The reflector $G \colon \mathbb B \rightarrow \mathbb A$ has \emph{stable units} \cite{CHK} if it preserves any pullback of the form
$$
\xymatrix{P \ar[r]^-{p_2} \ar[d]_{p_1} & X \ar[d]^{\phi}  \\
A \ar[r]_-{{\overline{\eta}}_A} & G(A),
}
$$
where $\overline{\eta}_A$ is the $A$-component of the unit ${\overline{\eta}}$ of the adjunction \eqref{semi-loc}.
This property can be seen to be equivalent to the property that the components of the unit are pullback-stable, in the sense that, for any square as above, one has an isomorphism $p_2 \cong \eta_P$. When this is the case, it follows in particular that $\mathbb A$ is a semi-localization of $\mathbb B$ \cite{Semi-loc}: indeed, by definition, for a reflector $G \colon \mathbb B \rightarrow \mathbb A$ having stable units is a stronger condition than being \emph{semi-left-exact} in the sense of \cite{CHK}.
The following well known lemma will be needed:
\begin{lemma}\label{pullback-mono}
Consider the following morphism of equivalence relations 
$$
\xymatrix{R \ar[r]^{\overline{f}} \ar@<-2pt>[d]_{r_1}  \ar@<2pt>[d]^{r_2}& S  \ar@<-2pt>[d]_{s_1}  \ar@<2pt>[d]^{s_2}\\
X \ar[r]_f  & Y}
$$
in an exact category $\mathbb C$. Then $f^{-1} (S) = R$ if and only if the unique induced arrow $\phi \colon X/R \rightarrow Y/S $ making the square 
$$
\xymatrix{X \ar[r]^f \ar@{->>}[d]_{q_R} & Y \ar@{->>}[d]^{q_S}  \\
X/R \ar@{.>}[r]_{\phi} & Y/S}
$$
commute is a monomorphism.
\end{lemma}
\begin{proof}
When $f^{-1} (S) = R$, the equivalence relation $(R, r_1, r_2)$ is the kernel pair of $q_S f$. In this case $\phi q_R$ is simply the usual factorization of $q_S f$ as a regular epimorphism followed by a monomorphism. Conversely, assume that $\phi$ is a monomorphism. If $(T, t_1, t_2)$ is another equivalence relation on $X$ and $(f, \hat{f}) \colon (T, t_1, t_2) \rightarrow (S, s_1, s_2)$ a morphism of equivalence relations, one sees that $q_R  t_1 = q_R t_2$ (since $\phi q_R  t_1 = \phi q_R t_2$). The universal property of the kernel pair $(R, r_1, r_2)$ of $q_R$ implies that there is a unique $\alpha \colon T \rightarrow R$ in $\mathbb C$ such that $ (f,\overline{f}) (1_X,\alpha)= (f,  \hat{f})$, as desired.
\end{proof}

\begin{theorem}\label{semi-localization}
Given any exact category $\mathbb C$, the reflector $F \colon {\mathsf{PreOrd}(\mathbb C) \, } \rightarrow \Par$ in the adjunction \eqref{adjunction}
has stable units. That is, the functor $F$ preserves any pullback in ${\mathsf{PreOrd}(\mathbb C)}$ of the form
 \begin{equation}\label{cube}\xymatrix@1@R=30pt@C=30pt{{\rho \times_{\pi_X(\rho)} \sigma} \ar@{->>}[rr]^{\overline{p}_2}
 \ar[dd]^{\overline{p}_1}\ar@<-2pt>[dr]_{l_1} \ar@<2pt>[dr]^{l_2}&& \sigma \ar[dd]^(.7){\overline{f}}
 \ar@<-2pt>[dr]_{s_1} \ar@<2pt>[dr]^{s_2}
 \\& {\quad X \times_Y Z}\ar@{->>}[rr]_(.7){p_2}
\ar[dd]^(.7){p_1}&&Z \ar[dd]^{f}\\ \rho  \ar@<-2pt>[dr]_{r_1} \ar@<2pt>[dr]^{r_2}  \ar@{->>}[rr]^(.7){\overline{\pi}_X} &&\pi_X(\rho)  \ar@<-2pt>[dr]_{t_1} \ar@<2pt>[dr]^{t_2}
 \\&X\ar@{->>}[rr]_{\pi_X}&&{\quad \frac{X}{\sim_{\rho}}} = Y}
 \end{equation}
 where the bottom square is the $(X, \rho)$-component of the unit of the adjunction \eqref{adjunction}. \end{theorem}
\begin{proof}
First note that $p_2$ and $\overline{p}_2$ in the cube \eqref{cube} are regular epimorphisms since both $\pi_X$ and $\overline{\pi}_X$ are regular epimorphisms.
The image of the cube \eqref{cube} by the reflector $F$ is the commutative diagram
 \begin{equation}\label{cube2}\xymatrix@1@R=30pt@C=30pt{\pi_{X \times_Y Z}({\rho \times_{\pi_X(\rho)} \sigma)} \ar@{->>}[rr]^{F(\overline{p}_2)}
 \ar[dd]^{F(\overline{p}_1)}\ar@<-2pt>[dr]_{m_1} \ar@<2pt>[dr]^{m_2}&& \pi_Z(\sigma) \ar[dd]^(.7){}
 \ar@<-2pt>[dr]_{s_1'} \ar@<2pt>[dr]^{s_2'}
 \\& {\quad \frac{X \times_Y Z}{\sim} }\ar@{->>}[rr]_(.8){F(p_2)}
\ar[dd]^(.7){F(p_1)}&&{\, \frac{Z}{\sim_{\sigma}}} \ar[dd]^{F(f)}\\ \pi_X(\rho)  \ar@<-2pt>[dr]_{t_1} \ar@<2pt>[dr]^{t_2}  \ar@{=}[rr]^(.7){1_{\pi_X(\rho)}} &&{\pi_X}(\rho)  \ar@<-2pt>[dr]_{t_1} \ar@<2pt>[dr]^{t_2}
 \\&{\, Y} \ar@{=}[rr]_{1_Y}&&{\,\,  Y} }
 \end{equation}
where we write $\sim$ instead of $ \sim_{{\rho \times_{\pi_X(\rho)} \sigma}}$ to simplify the notations. Note that $F(p_2)$ and $F(\overline{p}_2)$ are both regular epimorphisms, since so are $p_2$ and $\overline{p}_2$. Accordingly, to prove the result, it will suffice to show that $F(p_2)$ and $F(\overline{p}_2)$ are monomorphisms (and then isomorphisms), since this will imply that both the back and the front faces in the cube \eqref{cube2} are pullbacks.
For this, consider the commutative diagram
$$\xymatrix@=40pt{ {\rho \times_{\pi(\rho)} \sigma} \ar[d]_{\langle l_1,l_2\rangle}   \ar[r]^-{\overline{p}_2 }  & \sigma 
\ar[d]^{\langle s_1, s_2\rangle}  \\
 (X \times_Y Z) \times  (X \times_Y Z) \ar[r]_-{p_2 \times p_2} & {Z \times Z,} &
}
$$
and observe that it is a pullback, i.e., that ${p_2}^{-1} (\sigma) =  {\rho \times_{\pi(\rho)} \sigma}$, since in the following diagram both the external rectangle and the right-hand square are pullbacks:
$$\xymatrix@=40pt{ {\rho \times_{\pi_X(\rho)} \sigma} \ar[r]^-{\langle l_1,l_2\rangle}   \ar[d]_-{\overline{p}_2 }  & (X \times_Y Z) \times  (X \times_Y Z) \ar[d]^{p_2 \times p_2 }  \ar[r]^-{p_1 \times p_1 } & X\times X \ar[d]^{{\pi}_X \times {\pi}_X } \\
 \sigma 
  \ar[r]_-{\langle s_1,s_2\rangle} & {Z \times Z}  \ar[r]_{f \times f} & Y\times Y.
}
$$
This implies that 
$$\xymatrix@=40pt{ {(\rho \times_{\pi_X(\rho)} \sigma)}^o \ar[d]_{\langle l_2,l_1\rangle}   \ar[r]^-{\overline{p}_2 }  & \sigma^o 
\ar[d]^{\langle s_2, s_1\rangle}  \\
 (X \times_Y Z) \times  (X \times_Y Z) \ar[r]_-{p_2 \times p_2} & {Z \times Z,} &
}
$$
is a also pullback and, consequently,  the commutative square  
$$\xymatrix@=40pt{ \sim \ar[d]_{}   \ar[r]^-{ }  & \sim_{\sigma} 
\ar[d]^{}  \\
 (X \times_Y Z) \times  (X \times_Y Z) \ar[r]_-{p_2 \times p_2} & {Z \times Z,} &
}
$$
is a pullback (since $\sim {} = (\rho \times_{\pi_X(\rho) \sigma)} \wedge  {(\rho \times_{\pi_X(\rho)} \sigma)}^o$ and $\sim_{\sigma} {}= \sigma \wedge \sigma^o$). This means that $p_2^{-1}( \sim_{\sigma}) ={} {\sim}$ and, by Lemma \ref{pullback-mono}, it follows that $F(p_2) \colon 
\frac{X \times_Y Z}{\sim} \rightarrow \frac{Z}{\sim_{\sigma}}$ is a monomorphism. It follows that the arrow $\langle s_1', s_2' \rangle F(\overline{p}_2) =F(p_2)\times F(p_2) (m_1,m_2)$ is a monomorphism, thus $F(\overline{p}_2)$ is a monomorphism, hence an isomorphism, as desired.
\end{proof}
We will write $\mathsf{Cat}(\mathbb C)$ for the category of internal categories and internal functors in $\mathbb C$.
This category contains the category ${\mathsf{PreOrd}(\mathbb C)}$ as a full subcategory. Xarez proved in \cite{Xarez} (Corollary 5.2) that the reflector 
\begin{equation*}
\xymatrix@=30pt{
{\mathsf{Cat}(\mathbb C) \, } \ar@<1ex>[r]_-{^{\perp}}^-{G} & {\mathsf{PreOrd}(\mathbb C)}
\ar@<1ex>[l]^-V }
 \end{equation*}
has stable units, under some suitable assumptions on $\mathbb C$ that are certainly satisfied when $\mathbb C$ is an exact category. Accordingly, we have the following
\begin{corollary}
When $\mathbb C$ is an exact category, the reflector $L \colon \mathsf{Cat}(\mathbb C) \rightarrow \Par$ in the adjunction
\begin{equation}\label{Cat}
\xymatrix@=30pt{
{\mathsf{Cat}(\mathbb C) \, } \ar@<1ex>[r]_-{^{\perp}}^-{L} & \Par
\ar@<1ex>[l]^-W  }
 \end{equation}
 has stable units.
\end{corollary}
\begin{proof}
The adjunction \eqref{Cat} can be decomposed as the composite adjunction
\begin{equation*}  
\xymatrix@=30pt{
{\mathsf{Cat}(\mathbb C) \, } \ar@<1ex>[r]_-{^{\perp}}^-{G} & {\mathsf{PreOrd}(\mathbb C)} \ar@<1ex>[l]^-V \ar@<1ex>[r]_-{^{\perp}}^-{F} &  {\Par ,}  \ar@<1ex>[l]^-U  }
 \end{equation*}
 where $V U = W$ and $L = FG$.
 Since the composite of two reflectors with stable units is itself a reflector with stable units, the result then follows from Corollary $5.2$ in \cite{Xarez} and from Theorem \ref{semi-localization} above.
\end{proof}
Since the reflector $F \colon {\mathsf{PreOrd}(\mathbb C)} \rightarrow \Par$ in \eqref{adjunction} has stable units, and is therefore semi-left-exact, the class $\mathcal E$ of morphisms inverted by the reflector $F$ in \eqref{adjunction} and the class $\mathcal M$, that is the closure under pullbacks of morphisms lying in $\Par$, determine a (reflective) factorization system $(\mathcal E, \mathcal M)$ in $\mathsf{PreOrd}(\mathbb C)$. This follows from the general theory of factorization systems \cite{CHK}. Thanks to the fact that the reflector $F$ is semi-left-exact, we know that the class $\mathcal E$ is stable under pullbacks along morphisms in $\mathcal M$, and this latter class admits a simple description. Indeed, the canonical $(\mathcal E, \mathcal M)$-factorization $m e$ of a morphism $f \colon (A, \rho) \rightarrow (B, \sigma)$ in  $\mathsf{PreOrd}(\mathbb C)$ can be easily obtained via the following pullback: 
$$\xymatrix@=40pt{A \ar[rrd]^f \ar[ddr]_{\pi_A} \ar@{.>}[dr]^{e \in \mathcal E} & & \\& F(A)\times_{F(B)} B \ar[d] \ar[r]_-{m \in \mathcal M} & B \ar[d]^{\pi_B} \\& F(A) \ar[r]_{F(f)} & F(B).
}
$$
The two classes $(\mathcal E, \mathcal M)$ in $\mathsf{PreOrd}(\mathbb C)$ can be explicitly described as follows:
\begin{proposition}\label{factor-E}
Given the adjunction \eqref{adjunction}, consider the associated factorization system $(\mathcal E, \mathcal M)$. Then
\begin{enumerate}
\item a morphism \begin{equation}\label{generic}
\xymatrix{   \rho  
\ar@<.5ex>[d]^{r_2} \ar@<-.5ex>[d]_{r_1}  \ar[r]^{\overline{f}}  & \sigma \ar@<.5ex>[d]^{s_2} \ar@<-.5ex>[d]_{s_1}  \\
   A \ar[r]_{f} & {B} &
}
\end{equation}
in $\mathsf{PreOrd}(\mathbb C)$ is in the class $\mathcal E$ if and only if  it is fully faithful, i.e. the square 
 \begin{equation}\label{pullback-fully-faithful} 
\xymatrix{\, {\rho}\ar[d]  \ar[r]^{\overline{f}} & {\, \, {\sigma}} \ar[d]
\\A \times A \ar[r]_{f \times f} & B \times B
}
\end{equation}
is a pullback and, moreover, the induced arrow $\frac{A}{\sim_{\rho}} \rightarrow \frac{B}{\sim_{\sigma}} $ is a regular epimorphism.
\item a morphism \eqref{generic} is in the class $\mathcal M$ if and only if both the commutative squares
$$\xymatrix{   \sim_{\rho}  
\ar@<.5ex>[d]^{} \ar@<-.5ex>[d]_{}  \ar[r]^{\hat{f}}  & \sim_{\sigma} \ar@<.5ex>[d]^{} \ar@<-.5ex>[d]_{}  \\
   A \ar[r]_{f} & {B} &
}$$
are pullbacks. This means that the internal functor $(f, \hat{f}) \colon (A, \sim_{\rho}) \rightarrow (B, \sim_{\sigma})$ is a \emph{discrete fibration}.
\end{enumerate}
\end{proposition}
\begin{proof} 
$(a)$ First observe that the assumption that the square \eqref{pullback-fully-faithful} is a pullback implies that the following square is also a pullback:
\begin{equation}\label{sigma-fully-faithful} 
\xymatrix{\, \sim_{\rho}\ar@{>->}[d]  \ar[r]^{\hat{f}} & {\, \, \sim_{\sigma}} \ar@{>->}[d]
\\A \times A \ar[r]_{f \times f} & B \times B
}
\end{equation}
Indeed, this easily follows from the construction of $\sim_{\rho}$ (and $\sim_{\sigma}$) as pullbacks (see \eqref{sim}).
Then, by applying Lemma \ref{pullback-mono} to the commutative diagram
\begin{equation}\label{lower}
\xymatrix@=35pt{A \ar[d]_{\pi_A}   \ar[r]^{f} & B \ar[d]^{\pi_B}
\\ \frac{A}{{\sim}_{\rho}} \ar[r]_{F(f)} & \frac{B}{{\sim}_{\sigma}}
}
\end{equation}
we see that $F(f)$ is a monomorphism if and only if the square \eqref{sigma-fully-faithful} is a pullback. It follows that $F(f)$ is an isomorphism if and only \eqref{sigma-fully-faithful} is a pullback and $F(f)$ is regular epimorphism. To complete the proof of $(a)$ observe that the internal functor \eqref{generic} is fully faithful if and only if the external rectangle in the following diagram is a pullback:
$$
\xymatrix@=55pt{\, {\rho}\ar[d]_{\langle r_1, r_2\rangle}  \ar[r]^{\overline{\pi}_A} & \pi_A(\rho)  \ar[d] \ar[r]^{F(\overline{f})} & \pi_A(\sigma) \ar[d] 
\\ 
A \times A \ar[r]_-{\pi_A \times \pi_A} & {\frac{A}{\sim_{\rho}} }  \times {\frac{A}{\sim_{\rho}} }  \ar[r]_{F(f) \times F(f)} & {\frac{A}{\sim_{\sigma}} }  \times {\frac{A}{\sim_{\sigma}} } 
}
$$
By taking into account that $\pi_A \times \pi_A$ is a regular epimorphism and that the left square is a pullback (see \eqref{direct-inverse}), we conclude by Proposition $2.7$ in \cite{JK} that this is also equivalent to the property that its right square is a pullback, hence to the fact that $F(\overline{f})$ is an isomorphism. 

 The condition $(b)$ is a consequence of the fact that $(f, \overline{f})$ is in $\mathcal M$ if and only if its naturality square is a pullback in $\mathsf{PreOrd}(\mathbb C)$. This latter condition is easily seen to be equivalent to the fact that the square \eqref{lower} is a pullback in $\mathbb C$.
\end{proof}
\subsection*{Goursat categories.}
A regular category $\mathbb C$ is a \emph{Goursat category} \cite{CKP} when, for any pair $R$ and $S$ of equivalence relations on the same object in $\C$, one has the equality $R \circ S \circ R = S \circ R \circ S$. In the varietal context this property becomes the so-called $3$-permutability of congruences, which is slightly more general than the well-known $2$-permutability property, and is usually referred to as the Mal'tsev property in categorical algebra. Examples of regular Goursat categories are provided by the varieties of groups, abelian groups, $R$-modules, associative algebras, Lie algebras, quasigroups, boolean algebras, implication algebras, Heyting algebras, and topological groups. When $\C$ is a Goursat category, then $\mathsf{PreOrd}(\mathbb C)$ is the category $\Eq$ of equivalence relations in $\C$, since any reflexive and transitive relation in $\C$ is symmetric \cite{MFRV}. The category $\Eq$ is itself regular \cite{GRT}, and its full subcategory $\Par$ of partial orders coincides with the full subcategory $\mathsf{Dis}(\mathbb C)$ of discrete equivalence relations in $\C$, since, for any partial order $(A, \rho)$, one has that $\rho = \rho \wedge \rho= \rho \wedge \rho^o = \Delta_A$. 
More generally, by taking into account the results in \cite{MFRV},  we have the following:
\begin{proposition}\label{Goursat}
In a Goursat category the forgetful functors $\Eq \rightarrow \mathsf{PreOrd}(\mathbb C)$ and $\mathsf{Dis}(\mathbb C) \rightarrow \Par $ are isomorphisms. More generally, this is the case in any $n$-permutable regular category (in the sense of \cite{CKP}).
In the exact Goursat case, the adjunction $\eqref{adjunction}$ becomes the adjunction
\begin{equation*}\label{adjunction2}
\xymatrix@=30pt{
{\mathsf{Eq}(\mathbb C) \, } \ar@<1ex>[r]_-{^{\perp}}^-{\pi_0} & {\, \mathsf{Dis}(\mathbb C) ,\, }
\ar@<1ex>[l]^U  }
 \end{equation*}
 where the left adjoint $\pi_0$ is the classical ``connected component functor'', sending an equivalence relation to the coequalizer of its projections. 
\end{proposition}
The reflector in this latter adjunction is known to be semi-left-exact \cite{CHK} (even when the base category is only assumed to be exact, see \cite{BournSemi}).

\section{Monotone-light factorization system for preordered sets.}
From now on we shall assume that ${\mathbb C}$ is the category $\mathsf{Set}$ of sets.
In this case we shall give a precise characterization of the class ${\mathcal M}^*$ of \emph{coverings} in the sense of categorical Galois theory \cite{Janelidze-Galois} with respect to the adjunction \eqref{adjunction}, and then prove the existence of a monotone-light factorization system in ${\mathsf{PreOrd}(\mathsf{Set})}$.

Recall that, given a full reflective subcategory $\mathbb A$ of a category $\mathbb B$ as in \eqref{semi-loc}
 with a semi-left-exact reflector $G \colon \mathbb B \rightarrow \mathbb A$, 
 a morphism $f \colon A \rightarrow B$ is a \emph{covering} with respect to the adjunction \eqref{semi-loc} when there is an \emph{effective descent morphism} $p \colon E \rightarrow B$ with the property that the first projection $p_1 \colon E \times_B A \rightarrow E$ in the pullback 
 \begin{equation}\label{pullback}
 \xymatrix{E \times_B A \ar[r]^-{p_2} \ar[d]_{p_1} & A \ar[d]^f \\
E \ar[r]_p & B
}
\end{equation}
 belongs to $\mathcal M$, i.e., $p_1$ is a \emph{trivial covering}.
In the category ${\mathsf{PreOrd}(\mathsf{Set})}$ the effective descent morphisms have been characterized as follows (see Proposition $3.4$ \cite{JS}):
\begin{lemma}\label{char.effective.preOrd}
A morphism
$$
\xymatrix{   \rho  
\ar@<.5ex>[d]^{r_2} \ar@<-.5ex>[d]_{r_1}  \ar[r]^{\overline{p}}  & \sigma \ar@<.5ex>[d]^{s_2} \ar@<-.5ex>[d]_{s_1}  \\
   E \ar[r]_{p} & {B} &
}
$$
is an \emph{effective descent morphism} in ${\mathsf{PreOrd}(\mathsf{Set})}$ if and only if, given elements $b_1, b_2, b_3 \in B$ such that $b_1 \sigma b_2$ and $b_2 \sigma b_3$, there are elements
$e_1, e_2, e_3$ in $E$ such that $e_1 \rho e_2$ and $e_2 \rho e_3$ with $p(e_i)= b_i$, for $i=1,2,3$.
\end{lemma}

\begin{proposition}\label{effective}
 Let $(B,\rho)$ be a preordered set. Then there exists a partially ordered set $(A,\le)$ and an effective descent morphism ${p}\colon (A,\le)\to(B,\rho)$ in $\mathsf{PreOrd}(\mathsf{Set})$.
\end{proposition}
\begin{proof} 
 Keeping in mind \cite[Proposition 2.2]{FF} or the proof of Lemma \ref{N-torsion}, the preorder $\rho$ induces an equivalence relation $\sim_\rho$ on $B$ and a partial order $\leq_\rho:=\pi(\rho)$ on the quotient set $B/\!\!\sim_\rho$, defined by setting $[b]\leq_\rho [b']$ if, and only if, $b\rho b'$. The canonical projection $\pi:(B,\rho)\to (B/\!\!\sim_\rho,\leq_\rho)$ is a morphism of preordered sets. Now, let $\{1,2,3\}$ be endowed with the usual order $\leq$, let $B$ be endowed with the discrete partial order $=$ and let $\Gamma:=(B/\!\!\sim_\rho)\times \{1,2,3\}\times B$, equipped with the lexicographic order $\preceq$ (it is partial order, because it is induced by the partial orders $\leq_\rho,\leq,=$). Consider the subset 
$$
A:=\{([b], i, \beta)\in \Gamma\mid \beta\in [b] \}
$$
of $\Gamma$, endowed with the partial order induced by $\preceq$, and let ${p}:(A, \preceq)\to (B,\rho)$ be the mapping such that $p([b],i,\beta):=\beta$, for any $([b],i,\beta)\in A$. We claim that $p$ is a morphism of preordered sets. Indeed, if $([b],i,\beta),([b'],i',\beta')\in A$ and $([b],i,\beta)\preceq ([b'], i', \beta')$, then we have two cases. If $[b]=[b']$, then $\beta \sim_\rho b\sim_\rho b'\sim_\rho \beta'$ and, a fortiori (since $\sim_\rho$ is coarser than $\rho$), $\beta\rho \beta'$. If $[b]\neq [b']$, then we must have $[b]<_\rho[b']$, and thus $b\rho b'$. Since $\beta \sim_\rho b$ and $b'\sim_\rho \beta'$, we infer again $\beta \rho \beta'$. By definition, $p$ is surjective. Now let $b_1,b_2,b_3\in B$ be such that $b_1 \rho b_2$ and $b_2\rho b_3$. For any $i\in\{1,2,3\}$, let $\zeta_i:=([b_i],i,b_i)\in A$, and note that, by definition, ${p}(\zeta_i)=b_i$. By Lemma \ref{char.effective.preOrd}, the conclusion will follow if we prove that $\zeta_1\preceq \zeta_2$ and that $\zeta_2\preceq \zeta_3$. We only prove that $\zeta_1\preceq\zeta_2$ (the other statement has a similar proof). In the case $[b_1]=[b_2]$, the inequality ${1\le2}$ implies $\zeta_1\preceq\zeta_2$. Suppose now $[b_1]\neq [b_2]$. Then $b_1\rho b_2$ implies $[b_1]{\le_\rho}[b_2]$ and thus $\zeta_1\preceq \zeta_2$. The proof is now complete. 
\end{proof}
The following remark will be useful:
\begin{lemma}\label{3-2}
Consider a morphism $$
\xymatrix{   \rho  
\ar@<.5ex>[d]^{r_2} \ar@<-.5ex>[d]_{r_1}  \ar[r]^{\overline{f}}  & \sigma \ar@<.5ex>[d]^{s_2} \ar@<-.5ex>[d]_{s_1}  \\
   A \ar[r]_{f} & {B} &
}
$$
in ${\mathsf{PreOrd}(\mathsf{Set})}$ such that $(B, \sigma)$ is a partial order, and any fibre $(f^{-1}(b), \rho)$ is a partial order (for any $b \in B$). Then $(A, \rho)$ is a partial order.
\end{lemma}
\begin{proof}
Consider two elements $a$, $a'$ in $A$ with the property that $a \rho a'$ and $a' \rho a$. One then has that $f(a) \sigma f(a')$ and $f(a') \sigma f(a)$, so that $f(a') = f(a)$ and then both $a$ and $a'$ belong to the same fibre $f^{-1}(f(a))$, which is a partial order by assumption, and $a =a'$.
\end{proof}
\begin{proposition}\label{character.covering}
A morphism $f \colon (A, \rho)  \rightarrow (B, \sigma)$ in $\mathsf{Preord(Set)}$ is a \emph{covering} with respect to the adjunction \eqref{adjunction} if and only if any fibre $(f^{-1}(b), \rho)$ belongs to $\mathsf{ParOrd(Set)}$ for every $ b \in B$.
\end{proposition}
\begin{proof}
Let us first prove that any covering $f \colon (A, \rho)  \rightarrow (B, \sigma)$ has its fibres in $\mathsf{ParOrd(Set)}$. Consider a pullback \eqref{pullback},
where $p$ is an effective descent morphism and $p_1$ a trivial covering. In particular the map $p$ is surjective (by Lemma \ref{char.effective.preOrd}): for each $b \in B$ there is an $e \in E$ with $p(e)=b$ and, moreover, $p_1^{-1}(e)\cong f^{-1}(b)$. Since $p_1$ is a trivial covering, the following square is a pullback:
\begin{equation*}\label{pullback-2}
 \xymatrix{E \times_B A \ar[r]^-{\eta_{E \times_B A}} \ar[d]_{p_1} & F(E \times_B A)  \ar[d]^{F(p_1)} \\
E \ar[r]_{\eta_E} & F(E)
}
\end{equation*}
Then, each fibre $f^{-1}(b)\cong p_1^{-1}(e)$ is also isomorphic to ${F(p_1)}^{-1}(\eta_E(e)) \in \mathsf{ParOrd(Set)}$.

Conversely, assume that each fibre $f^{-1}(b)$ belongs $\mathsf{ParOrd(Set)}$, for $b \in B$. By Proposition \ref{effective} one can then ``cover'' $B$ with an effective descent morphism $p\colon E \rightarrow B$, with the domain $E$ in $\mathsf{ParOrd(Set)}$, and build the pullback \eqref{pullback}. We see that $p_1$ satisfies the assumptions of Lemma \ref{3-2}, so that $p_1 \colon E \times_B A \rightarrow E$ lies in $\mathsf{ParOrd(Set)}$, and $f$ is a covering, as desired.
\end{proof}
Note that the characterization of the coverings given in the proposition above is the same as the one of the so-called \emph{locally semi-simple coverings} in \cite{JMT2}. 

We are then ready to prove the following 
\begin{corollary}{}\label{characterization-N-k}
A morphism $f \colon (A, \rho)  \rightarrow (B, \sigma)$ in $\mathsf{Preord(Set)}$ is a \emph{covering} with respect to the adjunction \eqref{adjunction} if and only if its $\N$-kernel ${\mathsf{Ker}}_{\N}(f)$ is a partial order.\end{corollary}
\begin{proof}
It suffices to check that any fibre $f^{-1}(b)$ is a partial order if and only if $${\mathsf{Ker}}_{\N}(f) = \rho \wedge Eq(f) \in \mathsf{ParOrd(Set)}.$$ Assume that ${\mathsf{Ker}}_{\N}(f) = \rho \wedge Eq(f) \in \mathsf{ParOrd(Set)}$ and consider $a,a'$ both in $f^{-1}(b)$, with $a \rho a'$ and $a' \rho a$. It follows that $f(a) =b= f(a')$, hence $(a,a')\in \rho \wedge Eq(f)$, so that $a=a'$.
Conversely, if any fibre $f^{-1}(b)$ belongs to $\mathsf{ParOrd(Set)}$, consider $(a,a') \in {\mathsf{Ker}}_{\N}(f)$ such that $a \rho a'$ and $a' \rho a$. Then $\{a, a' \} \subset f^{-1}(f(a))$, and $a=a'$ because $f^{-1}(f(a))$ is a partial order.
\end{proof}
We shall also prove that there is a monotone-light factorization system $({\mathcal E}', {\mathcal M}^*)$ (in the sense of \cite{CJKP}) in ${\mathsf{PreOrd}(\mathsf{Set})}$, where ${\mathcal E}'$ is the class of morphisms that are stably in $\mathcal E$.
 
\begin{lemma}\label{characterization-bar-E}
The following conditions are equivalent for a surjective morphism \\ $f \colon (A, \rho)  \rightarrow (B, \sigma)$ :
\begin{enumerate}
	\item $f$ is fully faithful.
	\item ${\mathsf{Ker}}_{\N}(f) = (A, Eq(f))$ and $f(\rho) = \sigma$. 
	\item $Eq(f)\subseteq \rho$ and $f$ is a regular epimorphism in $\mathsf{PreOrd(Set)}$. 
\end{enumerate}
\end{lemma}
\begin{proof}
$(a)\Longrightarrow (b)$.  When $f \colon (A, \rho)  \rightarrow (B, \sigma)$ is fully faithful, we have $${\mathsf{Ker}}_{\N}(f)= (A, Eq(f) \wedge \rho) = (A, f^{-1} (\Delta_B) \wedge f^{-1} (\sigma)) = (A, f^{-1} (\Delta_B) )= (A, Eq(f)).$$ 
$(b)\Longrightarrow (a)$.  If $Eq(f) = Eq(f) \wedge \rho$ and $\sigma = f(\rho) = f \circ \rho \circ f^o$, we get that $$f^{-1}(\sigma) = f^o \circ \sigma \circ f =  f^o \circ (f \circ \rho \circ f^o) \circ f = Eq(f) \circ \rho \circ Eq(f) \le \rho,$$
and thus $f^{-1}(\sigma) = \rho$, as desired. 

Finally, the equivalence of $(b)$ and $(c)$ is a straightforward consequence of \cite[Proposition 2.2]{JS}
\end{proof}
 \begin{proposition}\label{canonical-factor}
 Any morphism $f \colon (A, \rho) \rightarrow (B, \sigma)$ in ${\mathsf{PreOrd}(\mathsf{Set})}$ has a canonical factorization 
 $$\xymatrix{(A, \rho) \ar[rr]^f \ar@{->>}[dr]_e & &  (B, \sigma) \\
&(\frac{A}{Eq(f) \wedge \sim_{\rho}}, e(\rho)) \ar[ru]_m &  }$$
where $e$ is the quotient, and $m \in {\mathcal M}^*$ is the unique morphism with $m e = f$. Furthermore, the surjective map $e \colon (A, \rho) \rightarrow (\frac{A}{Eq(f) \wedge \sim_{\rho}}, e(\rho))$ is a fully faithful functor. 
 \end{proposition}
\begin{proof}
First observe that the map $e$ is surjective by construction, and consider the commutative diagram 
$$
\xymatrix{Eq(f) \wedge {\sim}_{\rho} \ar[r] \ar@<.5ex>[d]^{f_2} \ar@<-.5ex>[d]_{f_1}& \rho \ar[r]^{\overline{e}} \ar@<.5ex>[d]^{r_2} \ar@<-.5ex>[d]_{r_1} & e(\rho) \ar@<.5ex>[d]^{s_2} \ar@<-.5ex>[d]_{s_1} \\
A \ar@{=}[r] & A \ar@{->>}[r]_-{e} & \frac{A}{Eq(f) \wedge \sim_{\rho}}.
}
$$
The fact that ${{\mathsf{Ker}}_{\N}(e) = } Eq(f) \wedge {\sim}_{\rho} \le \rho$ implies that $e$ is fully faithful (by Lemma \ref{characterization-bar-E}, since $\overline{e}$ is surjective).

On the other hand, one can easily check that the square 
$$
\xymatrix{Eq(f) \wedge \rho \ar[r] \ar@{.>}[d]^{\varepsilon} & \rho \ar@{->>}[d]^{\overline{e}} \\
Eq(m) \wedge e(\rho)\ar[r] & e(\rho) }
$$
is a pullback, where the horizontal morphisms are the canonical inclusions, and $\varepsilon$ is the restriction of $\overline{e}$ to $Eq(f) \wedge \rho$. This implies that $\varepsilon$ is surjective, and that \begin{eqnarray}
 {\mathsf{Ker}}_{\N}(m) & = & (\frac{A}{Eq(f) \wedge \sim_{\rho}) }, Eq(m) \wedge e(\rho) ) \nonumber \\ & = & (\frac{A}{Eq(f) \wedge \sim_{\rho}) }, e(Eq(f) \wedge \rho)) \nonumber \\  & = & F({\mathsf{Ker}}_{\N}(f)) \in {\mathsf{ParOrd}(\mathsf{Set})}, \nonumber \end{eqnarray}
where $F \colon {\mathsf{PreOrd}(\mathsf{Set})} \rightarrow {\mathsf{ParOrd}(\mathsf{Set})}$ is the reflector.
It follows from Corollary \ref{characterization-N-k} that $m \in {\mathcal M}^*$, as desired.
\end{proof}

Let us denote by $\overline{\mathcal E}$ the class of \emph{surjective} maps $f \colon (A, \rho)  \rightarrow (B, \sigma)$ in $\mathsf{Preord(Set)}$ that are \emph{fully faithful}. 
\begin{proposition}\label{monotone-light}
The adjunction 
\begin{equation}\label{fund-set}
\xymatrix@=30pt{
{\mathsf{PreOrd}(\mathsf{Set}) \, } \ar@<1ex>[r]_-{^{\perp}}^-{F} & {\, \mathsf{ParOrd}(\mathsf{Set}) ,\, }
\ar@<1ex>[l]^U  }
 \end{equation}
induces the monotone-light factorization system $(\overline{\mathcal E}, {\mathcal M}^*)$.
\end{proposition} 
\begin{proof}
To prove that $\overline{\mathcal E} \subset  {\mathcal E}'$, consider a map $e \colon (A, \rho) \rightarrow (B, \sigma)$ in $\overline{\mathcal E}$. 
By Proposition \ref{factor-E} $(a)$ we know that $e$ belongs to $\mathcal E$: indeed, $e$ is fully faithful by definition of $\overline{\mathcal E}$, and $ F(e): \frac{A}{\sim_{\rho}} \rightarrow \frac{B}{\sim_{\sigma}}$ is clearly surjective, because $e$ is surjective.
Furthermore, it is easy to see that the maps in $\overline{\mathcal E}$ are stable under pullbacks: indeed, fully faithful functors between preordered sets are pullback-stable, and so are surjective maps in $\mathsf{Set}$. It follows that $\overline{\mathcal E} \subset  {\mathcal E}'$. 

We then remark that $(\overline{\mathcal E},{\mathcal M}^*)$ is a factorization system. Indeed, in Proposition \ref{canonical-factor} we have shown the existence of an $(\overline{\mathcal E}, {\mathcal M}^*)$-factorization of any morphism in $\mathbb C$. Clearly, both the classes $\overline{\mathcal E}$ and ${\mathcal M}^*$ are stable by composition with isomorphisms. Let us then show that these two classes of morphisms are orthogonal. Consider a commutative square
$$\xymatrix{(A, \rho) \ar[r]^e \ar[d]_u & (B,\sigma) \ar[d]^v \\
(C, \tau) \ar[r]_m & (D, \psi) 
}
$$
where $e \in \overline{\mathcal E}$ and $m \in {\mathcal M}^*$. Since ${\mathsf{Ker}}_{\N}(e) = (A, Eq(e))$ (Lemma \ref{characterization-bar-E}) and ${\mathsf{Ker}}_{\N}(m) \in \mathsf{ParOrd(Set)}$ (Corollary \ref{characterization-N-k}), we deduce that $u \colon (A, \rho) \rightarrow (C, \tau)$ coequalizes the projections $e_1 \colon Eq(e) \rightarrow A$ and $e_2 \colon Eq(e) \rightarrow A$ of the kernel pair of $e$. It follows that there is a unique $\alpha \colon (B, \sigma) \rightarrow (C, \tau)$ such that $\alpha e = u$, and this morphism is also such that $m \alpha = v$.

The class ${\mathcal E' }$ is always orthogonal to the morphisms in ${\mathcal M}^*$ (by Proposition $6.7$ in \cite{CJKP}). By taking into account that $(\overline{\mathcal E},{\mathcal M}^*)$ is a factorization system, we deduce that 
${\mathcal E' } \subset \overline{\mathcal E}$, and therefore ${\mathcal E' } = \overline{\mathcal E}$ as desired.
\end{proof}

\begin{remark}
\emph{It would be interesting to extend Proposition \ref{monotone-light} to a more general context which would include the case of $\mathsf{PreOrd} (\mathsf{Set})$ as a special case. As a first step in this direction it would be useful to generalize Proposition \ref{effective} to the category $\mathsf{Preord(\mathbb{C})}$, where $\mathbb C$ is an exact category satisfying suitable additional conditions. We intend to investigate this question in the future.}
\end{remark}

\subsection*{Topological interpretation.}
In what follows we will provide a description, in topological terms, of the morphisms in the classes $\cal E'$ and in $\cal M^*$ determined by the adjunction \eqref{fund-set}. In order to do this, it is useful to recall some topological preliminaries, for the reader's convenience. A topological space is called  \emph{Alexandroff-discrete} provided that the intersection of any non-empty collection of open sets is open. As it is well-known, the subcategory $\Alex$ of $\mathsf{Top}$ whose objects are Alexandroff-discrete spaces is canonically isomorphic to the category $\mathsf{Preord(Set)}$ of preordered sets (whose morphisms are clearly all monotone mappings). More precisely, given any preordered set $(X,\rho)$, consider the topology $\mathcal T_\rho$ on $X$ consisting of all subsets $A$ of $X$ such that, whenever $a\in A$ and $x\in X$ is such that $x\rho a$, then $x\in A$. Then it is easily seen that $(X,\mathcal T_\rho)$ is an Alexandroff-discrete space. For every $x_0\in X$, let 
$$\{x_0\}^{\leftarrow_\rho}:=\{x\in X\mid x\rho x_0 \}, \qquad \{x_0\}^{\to_\rho}:=\{x\in X\mid x_0\rho x\}.$$ Then, $\{x_0\}^{\leftarrow_\rho}$ is open in $X$, equipped with the Alexandroff-discrete topology $\mathcal T_\rho$ canonically associated to $\rho$. More precisely, $\{x_0\}^{\leftarrow_\rho}$ is the intersection of all open sets of $(X,\cal T_\rho)$ containing $x_0$. Dually, $\{x_0\}^{\to_\rho}$ is the closure of $\{x_0\}$. Moreover, a mapping $f:(X,\rho)\to (Y,\sigma)$ of preordered sets is a morphism in $\mathsf{Preord(Set)}$ if and only if $f$ is continuous, as a mapping $(X,\mathcal T_\rho)\to (X,\mathcal T_\sigma)$.

Now, let $(X,\mathcal T)$ be any topological space. Then, $\mathcal T$ induces a natural preorder $\leq_{\cal T}$ on $X$ defined by setting, for every $x,y\in X$, $x\leq_{\cal T}y:\!\!\iff y\in \overline{\{x\}}$. If $\rho$ is any preorder on $X$, then $\leq_{\cal T_{\rho}}=\rho$. Conversely, given any Alexandroff-discrete topology $\cal A$, then $\cal A=\cal T_{\leq_{\cal A}}$. The previous arguments show the well-known fact that the categories $\mathsf{Alex}$ and $ \mathsf{Preord(Set)}$ are isomorphic.

We state now a straightforward consequence of the discussion above. Recall that a topology on a set $X$ is a \emph{partition topology} if every open set of the topology is clopen (that is, there is a partition of the space $X$ that is a basis for the topology). 
\begin{lemma}\label{partition}
	Let $X$ be a topological space and let $\rho$ denote the preorder induced by the topology of $X$. 
	\begin{enumerate}
		\item $X$ is a T$_0$ space if and only if $\rho$ is a partial order. 
		\item If $X$ is Alexandroff-discrete, then the topology of $X$ is a partition topology if and only if $\rho$ is an equivalence relation. 
	\end{enumerate}
\end{lemma}
\begin{proof}
	By definition, $x\rho y$ if $y\in \overline{\{x\}}$. Thus (a) is clear (and well-known).
	
	(b). Assume that $X$ has a partition topology, let $x,y\in X$ be such that $y\in\overline{\{x\}}$, and take an open (and, by assumption, clopen) neighborhood $\Omega$ of $x$. Then $y\in \overline{\{x\}}\subseteq \overline{\Omega}=\Omega$, proving that  $x\in \overline{\{y\}}$ (notice that the assumption that $X$ is Alexandroff-discrete is not needed for this implication). Conversely, assume that $\rho$ is an equivalence relation and notice that, since $X$ is an Alexandroff-discrete space, $\{x\}^{\leftarrow_\rho}$ is an open set (that coincides with the equivalence class of $x$), for every $x\in X$. Now it is clear that each open set is the union of the equivalence classes of its elements. It follows that the partition determined by $\rho$ is a basis for the given Alexandroff-discrete topology. 
\end{proof}
It is worth noting that in case $X$ is not Alexandroff-discrete and the preorder induced by the topology is an equivalence relation, the topology can fail to be a partition topology: take for instance a set with at least two points equipped with  any non trivial, connected and T$_1$ topology.

In view of Lemma \ref{partition} and  the isomorphism between $\mathsf{Preord(Set)}$ and $\Alex$, the following topological counterpart of Lemma \ref{N-torsion} and Corollary \ref{adj-pre-par} is now clear. 
\begin{proposition}
	Let $\KoAlex $ (resp., $\PartAlex$) be the full subcategory of $\Alex$ whose objects are the T$_0$ Alexandroff-discrete spaces  (resp., the Alexandroff-discrete spaces with a partition topology). 
	\begin{enumerate}
		\item $(\PartAlex,\KoAlex)$ is a pretorsion theory in $\Alex$. 
		\item{ $\KoAlex$} is a (regular epi-)reflective subcategory of $\Alex$. More precisely there is an adjunction 
		\begin{equation*}\label{adj-alex}
		\xymatrix@=30pt{
			{\Alex \, } \ar@<1ex>[r]_-{^{\perp}}^-{F} & {\, \KoAlex ,\, }
			\ar@<1ex>[l]^U  }
		\end{equation*}
		where $U$ is the forgetful functor and $F$ is its left adjoint. For any object $X$ in $\Alex$, the $X$-component of the unit of the adjunction is the canonical projection $\pi_X:X\to X_0$, where $X_0$ is the T$_0$ quotient of $X$ (that is, $X_0$ is the quotient space of $X$ obtained by identifying points that have the same closure). 
	\end{enumerate}

\end{proposition}
\begin{remark}
	\emph{It is worth noting that in \cite[6.6]{pretorsion} a more general result is provided: if $\mathsf{Part}$ (resp., $\mathsf{T_0}$) are the full subcategories of $\mathsf{Top}$ whose objects are spaces equipped with a partition topology (resp., $T_0$ spaces), then $(\mathsf{Part},\mathsf{T_0})$ is a pretorsion theory in $\mathsf{Top}$. }
\end{remark}

The following fact is a topological form of \cite[Proposition 2.2]{JS}.
\begin{lemma}\label{top-regular-epi}
	Let $f\colon X\to Y$ be a surjective morphism in $\Alex$. Then the following conditions are equivalent.
	\begin{enumerate}
		\item $f$ is a regular epimorphism (in $\Alex$). 
		\item  The topology of $Y$ is the finest Alexandroff-discrete topology on $Y$ making $f:X\to Y$ a continuous mapping. 
	\end{enumerate}
\end{lemma}
Keeping in mind Proposition \ref{monotone-light} and the discussion above, there is a monotone-light factorization system $(\cal E',\cal M^*)$  in the category $\Alex$. The next goal is to describe it. 
\begin{proposition}\label{monotone-light-top}
	Let $(\cal E', \cal M^*)$ be the monotone-light factorization system in $\Alex$ induced by the natural  adjunction  (\ref{adj-alex})
	$$ \xymatrix@=30pt{
		{\Alex \, } \ar@<1ex>[r]_-{^{\perp}}^-{F} & {\, \KoAlex ,\, }
		\ar@<1ex>[l]^U  }$$

	and let $f:A\to B$ be a morphism in $\Alex$. 
	\begin{enumerate}
		\item $f\in \cal M^*$ if and only if each fibre of $f$ is T$_0$, as a subspace of $A$. 
		\item $f\in \cal E'$ if and only if the following conditions are satisfied:
		\begin{itemize}
			\item $f$ is surjective. 
			\item The topology of $B$  is the finest Alexandroff-discrete topology on $B$ making $f:A\to B$ a morphism in $\Alex$. 
			\item Each fibre of $f$ has the trivial topology.  
		\end{itemize}
	\end{enumerate}
\end{proposition}
\begin{proof}
	Statement (a) immediately follows from Proposition \ref{character.covering} and Lemma \ref{partition}(a). 
	
	(b). By Lemma \ref{characterization-bar-E} and Proposition \ref{monotone-light}, the class $\cal E'$ consists of all regular epimorphism $f:(A,\rho)\to (B,\sigma)$ satisfying $Eq(f)\subseteq \rho$. By Lemma \ref{top-regular-epi}, the first two conditions of statement (b) are equivalent to saying that $f$ is a regular epimorphism. Thus, it is sufficient to prove that a morphism $f:(A,\rho)\to (B,\sigma)$ satisfies $Eq(f)\subseteq \rho$ if and only if each fiber of $f$ has the trivial subspace topology. Suppose $Eq(f)\subseteq \rho$, fix an element $b\in B$ and let $\Omega$ be an open subset of $A$ such that $V:=\Omega\cap f^{-1}(b)$ is non-empty. Pick an element $v\in V$ and fix some $x\in f^{-1}(b)$. In particular, $f(x)=b=f(v)$, that is $x Eq(f) v$. Since, by assumption, $\rho $ is finer than $Eq(f)$, it follows $x\rho v$, in particular, and since $v\in \Omega$ and $\Omega$ is open we infer $x\in \Omega$, that is $x\in V$, proving that $V=f^{-1}(\Omega)$.
	
	Conversely, assume that each fibre of $f$ has the trivial topology, and take elements $x,y\in A$ such that $f(x)=f(y)=:b$. If $x \!\!\!\not\!\!\rho y$, $U:=f^{-1}(b)\setminus \overline{\{x\}}$ is a non-empty open subspace of $f^{-1}(b)$ (since $y\in U$) and thus, by assumption, $U=f^{-1}(b)$, against the fact that $x\in f^{-1}(b)\setminus U$. The conclusion is now clear. 
\end{proof}

\end{document}